\documentclass[12pt]{amsart}

\usepackage{amsmath,amsfonts,amssymb,amsthm}
\usepackage{pinlabel}
\usepackage{graphicx}
\usepackage[all]{xy}
\usepackage{hyperref}
\xyoption{dvips}

\numberwithin{equation}{section}

\newtheorem{thm}{Theorem}[section]
\newtheorem{cor}[thm]{Corollary}
\newtheorem{prop}[thm]{Proposition}
\newtheorem{lem}[thm]{Lemma}

\theoremstyle{remark}
\newtheorem{rem}[thm]{Remark}
\theoremstyle{example}

\theoremstyle{definition}
\newtheorem{defn}[thm]{Definition}

\theoremstyle{remark}
\newtheorem{remark}[thm]{Remark}

\newcommand{\R}{\mathbb{R}}

\newcommand{\Z}{\mathbb{Z}}
\newcommand{\C}{\mathbb{C}}

\newcommand{\F}{\mathbb{F}}

\newcommand{\ve}{\varepsilon}

\newcommand{\co}{\mskip0.5mu\colon\thinspace}

\begin{document}

\title{On homological rigidity and flexibility of exact Lagrangian endocobordisms}

\author{Georgios Dimitroglou Rizell}
\author{Roman Golovko}

\begin{abstract}
We show that an exact Lagrangian cobordism $L\subset \R \times P
\times \R$ from a Legendrian submanifold $\Lambda\subset P\times
\R$ to itself satisfies $H_i(L;\F)=H_i(\Lambda;\F)$ for any field
$\F$ in the case when $\Lambda$ admits a spin exact Lagrangian
filling and the concatenation of any spin exact Lagrangian filling
of $\Lambda$ and $L$ is also spin. The main tool used is Seidel's
isomorphism in wrapped Floer homology. In contrast to that, for
loose Legendrian submanifolds of $\C^n \times \R$, we construct
examples of such cobordisms whose homology groups have arbitrary
high ranks. In addition, we prove that the front $S^m$-spinning
construction preserves looseness, which implies certain
forgetfulness properties of it.
\end{abstract}

\address{Universit\'{e} Paris-Sud, D\'{e}partement de Math\'{e}matiques, Bat. 425, 91405
Orsay, France}
\email{georgios.dimitroglou@math.u-psud.fr}
\urladdr{https://sites.google.com/site/georgiosdimitroglourizell/}
\address{Universit\'{e} Paris-Sud, D\'{e}partement de Math\'{e}matiques, Bat. 425, 91405
Orsay, France} \email{roman.golovko@math.u-psud.fr}
\urladdr{http://www.math.u-psud.fr/~golovko/‎}
\date{\today}
\thanks{This work was partially supported by the ERC Starting Grant of Fr\'{e}d\'{e}ric Bourgeois StG-239781-ContactMath.}


\maketitle

\section{Introduction}

\subsection{Main Definitions}
Let $P$ be an exact symplectic $2n$-manifold  with symplectic form $\omega=d\theta$. The {\em contactization of $P$} is the product manifold $P\times \R$ equipped with a contact $1$-form $\alpha:=dz+\theta$, where $z$ is a coordinate on $\R$.
The {\em Reeb vector field of $(P\times \R, \alpha)$} is defined by $\alpha(R_{\alpha})=1$ and $i_{R_{\alpha}}d\alpha= 0$ and is in this case equal to $\partial_{z}$. We call $(\R\times P\times \R, d(e^{t}\alpha))$ the {\em symplectization} of $P\times \R$, where $t$ is a coordinate on the first $\R$-factor. This is again an exact symplectic manifold.

A {\em Legendrian submanifold} of $P\times \R$ is a submanifold of
dimension $n$ which is everywhere tangent to $\ker \alpha$. Observe
that there is a natural projection map $\Pi_{L}:P\times \R\to P$
which is called the {\em Lagrangian projection}. Any Legendrian
embedding can be $C^\infty$-approximated by a Legendrian embedding
$\Lambda$ such that the self-intersections of $\Pi_{L}(\Lambda)$
consist of a finite number of transverse double points. We call
Legendrian submanifolds which satisfy this property {\em chord
generic}. From now on we assume that all Legendrian submanifolds of
$P\times \R$ are closed, orientable, connected, and chord generic.
There is a one-to-one correspondence between the set of double
points of $\Pi_{L}(\Lambda)$ and the set of non-trivial integral
curves of the Reeb vector field $\partial_{z}$ which begin and end
on $\Lambda$. These integral curves are called {\em Reeb chords} of
$\Lambda$ and the set of Reeb chords will be denoted by $\mathcal
R(\Lambda)$. When $P=\C^n$ and $\theta=-\sum y_idx_i$, there is
another natural projection $\Pi_{F}:\C^{n}\times \R\to \R^{n+1}$
defined by $\Pi_{F}((x_{1},y_{1},\dots,z))=(x_{1},\dots,x_{n},z)$
and which is called the {\em front projection}. Observe that
Legendrian submanifolds can be recovered from their front
projections.


{\em Legendrian contact homology} is a part of the Symplectic Field
Theory framework \cite{ITSFT}. It was first defined by Chekanov
\cite{DAOLL} for knots in the standard contact $\C\times \R$ and
then extended to higher-dimensional contact manifolds such as
$\C^{n}\times \R$ \cite{TCHOLSIR} and contactizations $P\times \R$
\cite{LCHIPR} satisfying some additional assumptions. The Legendrian contact homology of $\Lambda$ is a
homology of the differential graded algebra $(\mathcal A(\Lambda),
\partial)$ which is a unital tensor algebra over the
$\Z_2$-vector space generated by the elements of $\mathcal
R(\Lambda)$ (we denote this vector space by $A(\Lambda)$) and is
graded using the Conley-Zehnder index. The differential $\partial$
counts pseudo-holomorphic curves in $(P,
d\theta)$ whose domains are disks with points removed on the
boundary. For more details we refer to \cite{LCHIPR}.

To make Legendrian contact homology well defined, we will in the following assume that $(P, d\theta)$ admits a proper smooth function $\psi \co P \to \R_{\ge 0}$ and a compatible almost complex structure $J_P$ such that $\psi$ is $J_P$-convex outside of some compact set $K$, i.e.~$-d(d\psi\circ J_P)(v,J_Pv)>0$ holds for $v\neq 0$ outside of $K$.

An {\em augmentation} $\ve$ of $\mathcal A(\Lambda)$ is a
differential graded algebra homomorphism from $(\mathcal A(\Lambda),
\partial)$ to  $(\Z_2, 0)$, i.e. $\varepsilon$ is a graded algebra map such that $\ve(1) = 1$ and $\ve\circ
\partial_{\Lambda} = 0$. Given an augmentation $\ve$, there is a way
due to Chekanov, see \cite{DAOLL}, to linearize $(\mathcal
A(\Lambda), \partial)$ to a finite dimensional differential graded
complex $CL_{\bullet}(\Lambda):=(A(\Lambda), \partial^{\ve})$. Here
$\partial^{\ve}$ is the linear part of $\widetilde{\partial}^{\ve}$
defined by $\widetilde{\partial}^{\ve}(c):=(\varphi^{\ve}\circ
\partial\circ (\varphi^{\ve})^{-1})(c)$, where $\varphi^{\ve}:
\mathcal A(\Lambda)\to \mathcal A(\Lambda)$ is a differential graded
algebdra isomorphism given by $\varphi^{\ve}(c): = c+\ve(c)$ on the
generators coming from $\mathcal R(\Lambda)$ and then extended as an
algebra map. We denote by $LCH_{\bullet}^{\ve}(\Lambda)$ the
homology of $CL_{\bullet}(\Lambda)$ and by
$LCH^{\bullet}_{\ve}(\Lambda)$ the homology of the dual complex
$CL^{\bullet}(\Lambda)$. Observe that not every Legendrian admits an
augmentation.

\begin{defn}
Let $\Lambda_{-}$ and $\Lambda_{+}$ be two Legendrian submanifolds of $P\times \R$.
We say that that $\Lambda_{-}$ is {\em exact Lagrangian cobordant to $\Lambda_{+}$} and write $\Lambda_{-}\prec^{ex}_{L}\Lambda_{+}$ if
there exists a smooth cobordism $(L; \Lambda_{-},
\Lambda_{+})$, and an exact Lagrangian embedding $L\hookrightarrow(\R\times P\times \R, d(e^t\alpha))$ satisfying
\begin{align*}
L|_{(-\infty, -T_{L}]\times P\times \R}& = (-\infty, -T_{L}]\times \Lambda_{-},\\
L|_{[T_{L},\infty)\times P\times \R}& = [T_{L},\infty)\times \Lambda_{+}
\end{align*}
for some $T_{L}\gg 0$ and where $L^{c}:=L|_{[-T_{L}-1,T_{L}+1]\times
P\times \R}$ is compact. We will in general not distinguish between
$L$ and $L^{c}$ and will use $L$ to denote both of them. If
$L_{\Lambda}$ is an exact Lagrangian filling of $\Lambda$ in the
symplectization of $P\times \R$, i.e., $L_{\Lambda}$ is an exact
Lagrangian cobordism with empty $-\infty$-boundary and whose
$+\infty$-boundary is equal to $\Lambda$, then we say that $\Lambda$
is {\em exact Lagrangian fillable} (or just {\em fillable}) and
write $\emptyset\prec^{ex}_{L_{\Lambda}}\Lambda$. In addition, if
there exists a spin exact Lagrangian filling of $\Lambda$, then we
say that $\Lambda$ is {\em spin fillable}. Finally, if
$\Lambda_{1}\prec^{ex}_{L} \Lambda_{2}$ and
$\Lambda_{2}\prec^{ex}_{L'} \Lambda_{3}$, then we denote by $L\ast
L'$ the exact Lagrangian cobordism obtained by gluing the positive
end of $L$ to the negative end of $L'$ so that
$\Lambda_{1}\prec^{ex}_{L\ast L'}\Lambda_{3}$.
\end{defn}

From now on we will assume that all exact Lagrangian cobordisms in the symplectization of $P\times \R$ are connected.

Ekholm in \cite{RSFTLLCHALFC} has shown that an exact Lagrangian
filling of a Legendrian submanifold of $P\times \R$ induces an
augmentation of its Legendrian contact homology algebra. More generally,
we should mention that given an exact Lagrangian cobordism $L$ with
$\Lambda_{-}\prec^{ex}_{L} \Lambda_{+}$ there is a differential
graded algebra morphism $\phi_L: \mathcal A(\Lambda_{+})\to \mathcal
A(\Lambda_{-})$ which is defined by a count of rigid punctured
pseudo-holomorphic disks in the symplectization of $P\times \R$ with boundary on $L$.

\begin{defn}
Given an exact Lagrangian cobordism $L$ from a
Legendrian submanifold $\Lambda\subset P\times \R$ to $\Lambda$, we
say that $L$ is an exact Lagrangian {\em endocobordism of
$\Lambda$}.
\end{defn}

\subsection{Results}
Rigidity of Lagrangian submanifolds has been discovered in many
situations. We refer to \cite{HOLSICB}, \cite{TSGOCBFACV} for the
case of Lagrangian submanifolds of a cotangent bundle and to
\cite{LCI}, \cite{ELCAP}, \cite{AFFLCTTFC} for the case of certain
Lagrangian cobordisms (observe that the notion of Lagrangian
cobordism in \cite{LCI}, \cite{ELCAP} and \cite{AFFLCTTFC} is
different from the one discussed here).

In this paper, we study exact Lagrangian endocobordisms. The results obtained here can be related to the results obtained in \cite{OANCEA} for subcritical fillings of contact manifolds. The main tool that will be used is the version of Seidel's isomorphism established in Theorem \ref{thm:isom}.

\subsubsection{Seidel's isomorphism for exact Lagrangian fillings}
There is an isomorphism described by Ekholm in~\cite{RSFTLLCHALFC}
which comes from a certain observation due to Seidel concerning the
wrapped Floer homology~\cite{AOSAOVF}, \cite{TSGOCBFACV}. The
following form of the isomorphism was proven by the first author
in~\cite{LPPTTSOPTRAA}.
\begin{thm}[Seidel]\label{malgthinstregbam}
Let $\Lambda$ be a Legendrian submanifold of $P\times \R$ and
$\emptyset\prec^{ex}_{L_{\Lambda}}\Lambda$. Then
\begin{align*}
H_{n-i}(L_{\Lambda}; \Z_{2})\simeq LCH_{\ve}^{i}(\Lambda; \Z_{2}).
\end{align*}
Here $\ve$ is the augmentation induced by $L_{\Lambda}$.
\end{thm}

Note that  the homology and cohomology groups in
Theorem~\ref{malgthinstregbam} are defined over $\Z_2$. To prove the
main rigidity result,  we need to use a slightly different version
of Seidel's isomorphism which holds for homology and cohomology
groups defined over a general unital ring $R$. To that end, in Section \ref{sec:wrapped} we define a version $CF^i_\infty (L_\Lambda;R)$ of the wrapped Floer homology complex associated to an exact Lagrangian filling. This complex is freely generated over $R$ by the Reeb chords on $\Lambda$, and its homology will be denoted by $HF^i_\infty (L_\Lambda;R)$. It is expected that this complex is isomorphic to $LCH_{\ve}^{i-1}(\Lambda;R)$ for some suitable choices, but we do not investigate this question. However, we establish the following result.
\begin{thm}[Seidel]
\label{thm:isom} Let $L_{\Lambda}$ be an exact Lagrangian filling which is
spin. For some suitable choices in the construction of the wrapped Floer homology complex of $L_\Lambda$, and after taking the grading to be modulo the Maslov number of $L_\Lambda$, we have an isomorphism
\[H_{n-i}(L_\Lambda;R) \simeq HF^{i+1}_\infty (L_\Lambda;R)\]
of homologies. In the non-spin case the same holds with $R=\Z_2$. In particular, the singular homology of an exact Lagrangian filling of $\Lambda$ can be computed by a complex generated by the Reeb chords on $\Lambda$.
\end{thm}
A more general version of this result appeared in \cite[Section 11]{RITTER}. Even though the latter version is formulated for closed contact manifolds, it can be adapted to work in the current setting as well. Also, we refer to \cite{OTLCBLS} for a version of this isomorphism obtained using the theory of generating families instead of pseudo-holomorphic curves.

Since we are interested in studying exact Lagrangian cobordisms, the
following long exact sequence obtained using Seidel's isomorphism
is important. This long exact sequence first appeared in the
work of the second author \cite{ANOLCBLSOR}, where it was shown for
the case $R=\Z_2$.
\begin{cor}
\label{cor:les} Let $L_\Lambda$ be an exact Lagrangian filling of
$\Lambda$ and let $L$ be an exact Lagrangian cobordism which can be
concatenated with $L_\Lambda$. Assume both $L_\Lambda$ and
$L_\Lambda * L$ to be spin cobordisms. There
is a long exact sequence
\[ \hdots \to H_i(\Lambda;R) \to HF^{n+1-i}_\infty (L_\Lambda;R) \oplus H_i(L;R) \to HF^{n+1-i}_\infty (L_\Lambda*L;R) \to \hdots \]
In the non-spin case, the same long exact sequence exists with $R=\Z_2$. Moreover, if we take the grading in the wrapped Floer homology groups to be induced by the grading in singular homology via Seidel's isomorphism, then the grading in the above long exact sequence may be taken to be in the integers.
\end{cor}
\begin{proof}
The long exact sequence is obtained from the Meyer-Vietoris long
exact sequence
\[ \hdots \to H_i(\Lambda;R) \to H_i(L_\Lambda;R) \oplus H_i(L;R) \to H_i(L_\Lambda*L;R) \to \hdots\]
with coefficients in $R$, after replacing the relevant terms using
the isomorphism in Theorem \ref{thm:isom}
\end{proof}

\subsubsection{Rigidity phenomena for endocobordisms}
We first prove the following rigidity result for fillable Legendrian submanifolds of $P\times \R$:
\begin{thm}\label{maintheoremsameranks}
Let $\Lambda$ be a spin fillable Legendrian submanifold of $P\times
\R$ and let $L$ be an exact Lagrangian endocobordism of $\Lambda$
inside the symplectization. In addition, assume that
$L_{\Lambda}\ast L$ is spin for any spin exact Lagrangian filling
$L_{\Lambda}$ of $\Lambda$. Then
\begin{itemize}
\item[(1)] $\dim H_{i}(L;\F) = \dim H_{i}(\Lambda;\F)$ for all i,
\item[(2)] the  map
\begin{align*}
(i^{-}_{*},i^{+}_{*}): H_{j}(\Lambda;\F)\to H_{j}(L;\F)\oplus
H_{j}(L;\F)
\end{align*}
is injective for all $j$.
\end{itemize}
Here $i^\pm$ is the inclusion of $\Lambda$ as the $\pm\infty$-boundary of $L$, and
$\F$ is an arbitrary field.
If $\F=\Z_2$, then the spin assumptions above can be omitted.
\end{thm}

The main ingredient used in the proof is Seidel's isomorphism.

\begin{remark}
Observe that we do not make any assumptions on the Maslov class of $L$ nor on $\Lambda$ in Theorem~\ref{maintheoremsameranks}
\end{remark}

\begin{remark}
The self-linking number of $\Lambda$ with itself pushed slightly in the Reeb direction is called the Thurston-Bennequin invariant, which is a Legendrian isotopy invariant. In the case when $P=\C^n$, an orientable Lagrangian cobordism $L$ from $\Lambda_-$ to $\Lambda_+$ satisfies
\[\mathrm{tb}(\Lambda_+)-\mathrm{tb}(\Lambda_-)=(-1)^{\frac{1}{2}(n^2-3n)}\chi(L,\Lambda_+).\]
We refer to \cite{LCOLK} for the proof in the case when $n=1$. The
proof of Chantraine can be naturally extended to the case of
arbitrary $n$. When $\Lambda_-=\Lambda_+$ we immediately conclude
that $\chi(L,\Lambda_+)=0$ and, in the case when $P=\C$, this actually implies Theorem~\ref{maintheoremsameranks} in the case of orientable $L$. However, when
$P\neq \C$, Theorem~\ref{maintheoremsameranks} can be seen as a
strengthening of this fact.
\end{remark}
Moreover, when $\Lambda$ is a homology sphere, we prove the
following theorem:

\begin{thm}\label{homologyspheremodginag}
\begin{enumerate}
\item If $\Lambda$ is a $\mathbb{Z}_2$-homology sphere admitting an exact Lagrangian filling, any exact Lagrangian endocobordism $L$ of $\Lambda$ has the property that the maps in homology
\[i^\pm_* \co H_j(\Lambda;\mathbb{Z}_2) \to  H_j(L;\mathbb{Z}_2)\]
induced by the inclusions of the boundary are isomorphisms. In
particular $L$ is spin and has vanishing Maslov class.
\item If $\Lambda$ is a $\mathbb{Z}$-homology sphere admitting a spin exact Lagrangian filling $L_\Lambda$, any exact Lagrangian endocobordism $L$ of $\Lambda$ satisfies the property that the maps in homology
\[i^\pm_* \co H_j(\Lambda;\mathbb{Z}) \to  H_j(L;\mathbb{Z})\]
induced by the inclusions of the boundary are isomorphisms.
\end{enumerate}
\end{thm}

It is natural to ask whether every exact Lagrangian cobordism
described in Theorem~\ref{maintheoremsameranks} is diffeomorphic to
$\R\times \Lambda$. In Section~\ref{sectiongeomaugmented}, we
construct an example of an exact Lagrangian endocobordism of a fillable Legendrian
$T^{2}\subset \C^{2}\times \R$ which is not diffeomorphic to
$\R\times T^{2}$.

\begin{remark}
Having an exact Lagrangian filling is known to have strong implications for the Legendrian submanifold. One important reason is that these fillings give rise to augmentations, which themselves have strong geometric implications (see e.g. \cite{ADESFLCH}). One could therefore expect that some analogue of
Theorem~\ref{maintheoremsameranks} holds in the more general case of
a Legendrian submanifold of $P\times \R$ whose Legendrian contact homology algebra admits an augmentaton. This
question is currently being studied by Chantraine, Ghiggini and the
authors, see \cite{FTFLC}.
\end{remark}

\subsubsection{Flexibility phenomena for endocobordisms}
\label{sec:flex}

Murphy in \cite{LLEIHDCM} has proven  an h-principle type result for
a class of Legendrian embeddings in contact manifolds of dimension
at least 5. These Legendrian submanifolds are called loose, see
Section~\ref{looselegsect} for the definition. In addition,
Eliashberg and Murphy in \cite{LC} have established an h-principle
for Lagrangian cobordisms with loose negative ends. We apply the result
of Eliashberg and Murphy and, in contrast to
Theorem~\ref{maintheoremsameranks}, get the following result for a
closed, connected genus $g$ loose Legendrian surface in
$\C^{2}\times\R$:
\begin{prop}\label{anycoblooselegsphere}
For any closed, orientable, connected $3$-manifold $M$ and any closed, connected loose Legendrian surface $\Sigma_g\subset \C^2\times\R$ of genus $g$, there is an
exact Lagrangian endocobordism $L$ of $\Sigma_g$ in the
symplectization of $\C^2 \times \R$ which is diffeomorphic to $M \#
(\R\times \Sigma_g)$.
\end{prop}
Similarly, we obtain the following result in higher dimensions.
\begin{prop}\label{bigspuntriviallooselegspheres}
For any loose Legendrian submanifold $\Lambda\subset \C^n \times \R$ and number $N>0$ there exists an exact
Lagrangian endocobordism $L$ of $\Lambda$ satisfying $\sum_i \dim
H_{i}(L;\F) \geq N$. Here $\F$ is an arbitrary field.
\end{prop}

\subsubsection{Looseness and the front spinning construction}
The second author defined the front $S^m$-spinning construction in
\cite{ANOTFSC} which, given a Legendrian embedding $\Lambda \subset \C^n\times \R$,
produces a Legendrian embedding $\Sigma_{S^m}\Lambda \subset \C^{n+m} \times \R$ of
the manifold $\Lambda \times S^m$. This generalises the construction described by Ekholm, Etnyre and
Sullivan in \cite{NILSIR}, which covers the case $m=1$. It was also shown that the spinning construction
extends to exact Lagrangian cobordisms. More precisely, given an exact Lagrangian cobordism $L$ from $\Lambda_-$ to $\Lambda_+$,
front spinning can be used to produce an exact Lagrangian cobordism $\Sigma_{S^m}L$ from $\Sigma_{S^m}\Lambda_-$ to $\Sigma_{S^m}\Lambda_+$ which is diffeomorphic to $L \times S^m$.

One can obviously apply the front spinning construction to the
examples in Section~\ref{sec:flex} to produce further examples.
However, not surprisingly, it can be shown that the front
$S^{m}$-spinning construction preserves looseness, see Appendix~\ref{fappthfrspconstrprlnss}.
\begin{prop}\label{spunlooselegisloose}
If $\Lambda$ is a loose Legendrian submanifold of $\C^{n}\times \R$, then $\Sigma_{S^m} \Lambda$
is a loose Legendrian submanifold of $\C^{n+m}\times \R$.
\end{prop}
In addition, it turns out that our proof of
Proposition~\ref{spunlooselegisloose} can be used to describe
certain forgetfulness properties of the front $S^{m}$-spinning
construction. More precisely, in Appendix~\ref{forgetfrontsmconstrex} we construct a few
examples of Legendrian submanifolds of $\C^n\times \R$ which are not
Legendrian isotopic but their front $S^{m}$-spuns become Legendrian
isotopic.

\section{Fillable Legendrians}\label{sectiongeomaugmented}

We  now prove Theorem~\ref{maintheoremsameranks}.
\subsection{Proof of Theorem~\ref{maintheoremsameranks}}
We start with the following relatively simple and purely topological
Lemma:
\begin{lem}\label{thomisomorphismmaininequality}
Let $L$ be a $(n+1)$-dimensional orientable endocobordism of $\Lambda$. Then
\begin{align}\label{ineqgoeintheformnat}
\sum\limits_{i} H_{i}(L;\mathbb F) \geq \sum\limits_{i}
H_{i}(\Lambda;\mathbb F).
\end{align}
Here $\mathbb F$ is an arbitrary field. Taking $\F=\Z_2$, the orientability assumption can be dropped.
\end{lem}
\begin{proof}
We first write the long exact sequence of $(L,\partial L)$

\[\xymatrix{\hdots \ar[r]  &H_{i}(\partial L;\mathbb F) \ar[r] & H_{i}(L; \mathbb F)  \ar[r] & H_{i}(L,\partial L; \mathbb F) \ar[r] & H_{i-1}(\partial L; \mathbb F)\ar[r] & \hdots}.\]
Observe that from the exactness it follows that
\begin{align*}
\dim H_{i}(L;\mathbb F)\geq \dim H_{i}(\partial L;\mathbb F) - \dim
H_{i+1}(L, \partial L;\mathbb F)=2\dim H_{i}(\Lambda) - \dim
H_{i+1}(L,
\partial L;\mathbb F).
\end{align*}
We apply Poincar\'{e} duality and rewrite the last inequality as
\begin{align*}
\dim H_{i}(L;\mathbb F)\geq 2\dim H_{i}(\Lambda; \mathbb F) - \dim
H^{n-i}(L;\mathbb F),
\end{align*}
which is equivalent to
\begin{align}\label{basineqforitosumandgetrfa}
\dim H_{i}(L;\mathbb F) + \dim H_{n-i}(L;\mathbb F) \geq 2\dim
H_{i}(\Lambda;\mathbb F).
\end{align}
Summing Formula~\ref{basineqforitosumandgetrfa} over all $i$ and
then dividing the result by $2$ leads us to
\begin{align*}
\sum\limits_{i} H_{i}(L;\mathbb F) \geq \sum\limits_{i}
H_{i}(\Lambda;\mathbb F).
\end{align*}
\end{proof}

We now prove that $\dim H_{i}(L;\F)\leq\dim H_{i}(\Lambda; \F)$ for
all $i$. Assume that there exists $i_0$ such that
\begin{align}\label{oppini0inintrnewf}
\dim H_{i_0}(L;\F)>\dim H_{i_0}(\Lambda; \F).
\end{align}

We apply the argument of the first part of Proposition 1.2 from
\cite{ANOTFSC}. From the Mayer-Vietoris long exact sequence for
$L_{\Lambda}\ast L \simeq L_{\Lambda} \cup L$ with $L_{\Lambda}\cap
L \simeq \R\times \Lambda$
\[\xymatrix{\hdots \ar[r]  &H_i(\Lambda;\F) \ar[r] & H_i(L_{\Lambda};\F)\oplus H_i(L;\F)  \ar[r] & H_i(L_{\Lambda}\ast L;\F) \ar[r] & \hdots}\]
it follows that
\begin{align*}
\dim(H_{i_0}(L_{\Lambda}\ast L;\F)) &\geq
\dim(H_{i_0}(L_{\Lambda};\F)) + \dim(H_{i_0}(L;\F)) -
\dim(H_{i_0}(\Lambda; \F)).
\end{align*}
We now use Formula~\ref{oppini0inintrnewf} and get that
\begin{align}\label{topolniceasforiodifs}
\dim(H_{i_0}(L_{\Lambda}\ast L;\F)) > \dim(H_{i_0}(L_{\Lambda};\F)).
\end{align}

Denote by $\mathcal L_{\mathrm{sp}}(\Lambda)$ the set of all spin embedded
exact Lagrangian fillings $L_\Lambda$ of $\Lambda$ and
\begin{align*}
\mathcal H_{\mathrm{sp}}(\Lambda):=\{ (\dim(H_{i}(L_{\Lambda}; \F)))_{i} \in \Z_{\ge 0}^{n+1}:
L_{\Lambda}\in \mathcal L_{\mathrm{sp}}(\Lambda)\}.
\end{align*}
Since the existence of Legendrian isotopy implies the existence of
an exact Lagrangian cylinder, see \cite{LCOLK}, \cite{ANOLCBLSOR} or
\cite{LITFDA}, it follows that $\mathcal H_{\mathrm{sp}}(\Lambda)$ is a
Legendrian invariant.

Note that for $L_{\Lambda}\in \mathcal L_{\mathrm{sp}}(\Lambda)$,
Theorem~\ref{thm:isom} 
implies
that
\begin{align}\label{maininemaxorclspcobnce}
\sum_{i} \dim (H_{i}(L_{\Lambda}; \F)) \leq c,
\end{align}
where $c$ is a number of Reeb chords of $\Lambda$. Note that for the last formula we do not need any assumptions on the Maslov class of $L_{\Lambda}$ or
the Maslov number of $\Lambda$.

From the fact that $\Lambda$ admits a spin exact Lagrangian filling
and Formula \ref{maininemaxorclspcobnce} it follows that there
exists a spin exact Lagrangian filling $L^{\mathrm{max}}_{\Lambda}\in
\mathcal L_{\mathrm{sp}}(\Lambda)$ such that
\begin{align}\label{maxiforhomolgybothends}
\dim(H_{i_{0}}(L^{\mathrm{max}}_{\Lambda}; \F))\geq
\dim(H_{i_{0}}(L_{\Lambda}; \F))
\end{align}
for all $L_{\Lambda}\in \mathcal L_{\mathrm{sp}}(\Lambda)$. Then we define
$L^{\mathrm{sep}}_{\Lambda}:= L_\Lambda^{\mathrm{max}}*L$. Observe that from the spin assumption of
the theorem it follows that $L^{\mathrm{sep}}_{\Lambda}\in \mathcal
L_{\mathrm{sp}}(\Lambda)$. Hence, Formula~\ref{topolniceasforiodifs} implies
that
\begin{align*}
\dim(H_{i_{0}}(L^{\mathrm{sep}}_{\Lambda};\F))>\dim(H_{i_{0}}(L^{\mathrm{max}}_{\Lambda};
\F))
\end{align*}
which contradicts Formula~\ref{maxiforhomolgybothends}. Therefore,
we get that
\begin{align}\label{nfthowaandnpoknp}
\dim H_{i}(L;\F)\leq\dim H_{i}(\Lambda; \F)
\end{align}
for all $i$. We now combine Formulas~\ref{ineqgoeintheformnat} and
\ref{nfthowaandnpoknp} and get that $\dim H_{i}(\Lambda;\F)=\dim
H_{i}(L; \F)$ for all $i$. This finishes the proof of the first part
of Theorem~\ref{maintheoremsameranks}.

For the second part, we start by considering the two exact Lagrangian endocobordisms $L$ and $L\ast L$. From the Mayer-Vietoris
long exact sequence for $L\ast L \simeq L \cup L$ with $L\cap L
\simeq \R\times \Lambda$
\[\xymatrix{\hdots \ar[r]  &H_j(\Lambda;\F) \ar[r]^--{(i^{-}_{\ast}, i^{+}_{\ast})} & H_j(L;\F)\oplus H_j(L;\F)
\ar[r] & H_j(L\ast L;\F) \ar[r] & \hdots}\]
it follows that
\begin{align}\label{inequalityposneendinj}
\dim H_j(L\ast L;\F) \geq 2\dim H_j(L;\F) - \dim (i^{-}_{\ast},
i^{+}_{\ast})(H_j(L;\F)).
\end{align}
In addition, we apply the first part of
Theorem~\ref{maintheoremsameranks} and get that
\begin{align*} \dim
H_j(L;\F) = \dim H_j(\Lambda;\F) = \dim H_j(L\ast L;\F)
\end{align*}
for all $j$. Hence, in order for Formula~\ref{inequalityposneendinj}
to hold, $\dim (i^{-}_{\ast}, i^{+}_{\ast})(H_j(L;\F)) = \dim
H_{j}(L;\F)$ and therefore $(i^{-}_{\ast}, i^{+}_{\ast})$ is
injective. This finishes the proof of
Theorem~\ref{maintheoremsameranks}.

\begin{remark}
Note that the condition on an exact Lagrangian endocobordism $L$
that the concatenation with it acts on $\mathcal L_{\mathrm{sp}}(\Lambda)$ is
not too restrictive. At least, in the following situations $L$
satisfies this property:
\begin{itemize}
\item[(i)] $L$ is orientable and $3$-dimensional;
\item[(ii)] $L$ has a Morse function whose gradient points inwards at the $-\infty$-boundary and outwards at the $+\infty$-boundary, all whose critical points are of index at least three. For instance, this is the case when $L = L_1\ast \dots\ast L_k$, where $L_i$ is an exact
Lagrangian elementary cobordism which corresponds to a Legendrian ambient $m_i$-surgery (see \cite{LASALCH}) for
$m_i\geq 2$ and $i=1,\dots,k$.
\end{itemize}
\end{remark}

\subsection{Proof of Theorem~\ref{homologyspheremodginag}}

We first prove the first part of
Theorem~\ref{homologyspheremodginag}. Since $L$ is connected,
Poincar\'{e} duality implies that
\[H_{n+1}(L,i_\pm(\Lambda);\Z_2)\simeq H^{n+1}(L,i_\pm(\Lambda);\Z_2) \simeq H_0(L,i_\mp(\Lambda);\Z_2) \simeq 0.\]
It follows that $i^\pm_*$ are injections in degree $n$ and
surjections in degree 0. Theorem~\ref{maintheoremsameranks} implies
that $i^\pm_*$ must be isomorphisms in all degrees.

We continue with the proof of the second part. Since $\Lambda$, in particular, is a $\mathbb{Z}_2$-homology
sphere, part (1) applies which, together with excision, shows that
the inclusion $i$ of $L_\Lambda$ into the concatenation $L_\Lambda *
L$ induces an isomorphism
\[i_* \co H_j (L_\Lambda;\mathbb{Z}_2) \to H_j (L_\Lambda*L;\mathbb{Z}_2)\]
of homology groups with $\mathbb{Z}_2$-coefficients. In conclusion,
since $i^*$ pulls back the second Stiefel-Whitney class of
$L_\Lambda*L$ to the second Stiefel-Whitney class of $L_\Lambda$, it
follows that $L_\Lambda * L$ is spin as well.

Theorem~\ref{maintheoremsameranks} thus applies for coefficients in
any field, and an argument analogous to the proof of part (1)
implies that $i^\pm_*$ are isomorphisms for coefficients in any
field.

In other words, the mapping cone $\mathrm{Cone}(i^\pm)$ of the chain map of
the inclusions of either boundary component is acyclic with
coefficients in any field. It follows that it is acyclic with
$\mathbb{Z}$-coefficients as well and hence that $i^\pm_*$ are
isomorphisms with $\mathbb{Z}$-coefficients.

\subsection{Non-cylindrical cobordism}

\begin{figure}[t]
\begin{center}
\labellist
\pinlabel $\Lambda''$ at 245 330
\pinlabel $\Lambda'$ at 245 180
\pinlabel $D_1$ at 20 175
\pinlabel $D_2$ at 135 92
\pinlabel $\Lambda$ at 245 30
\pinlabel $\Lambda$ at 510 330
\endlabellist
\includegraphics[width=315pt]{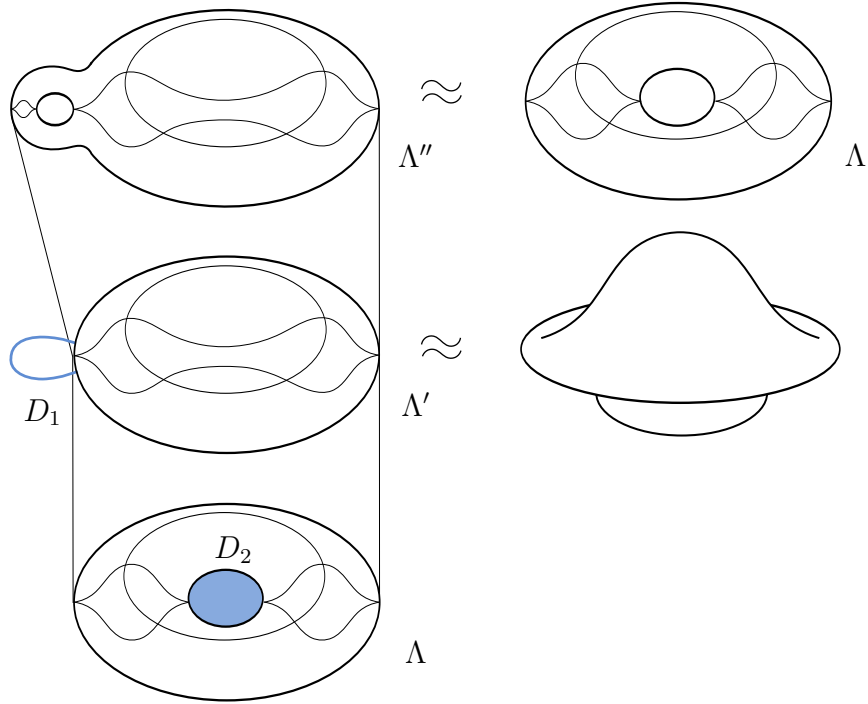}
\caption{The exact Lagrangiann cobordism not diffeomorphic to the cylinder.}\label{nontrcob}
\end{center}
\end{figure}

Let $\Lambda \simeq T^2$ be the Legendrian torus in $\C^{2}\times \R$ obtained as a result of the front spinning construction applied to the $tb=-1$ unknot $K$ in $\C\times \R$ (see Figure~\ref{nontrcob} for its front projection). From \cite{LLTKAT} (see also \cite{LCOLK}) it follows that $K$ is fillable and hence \cite[Proposition 1.5]{ANOLCBLSOR} implies that $\Lambda$ is fillable as well.

We construct the non-cylindrical exact Lagrangian endocobordism of $\Lambda$ using two Legendrian ambient surgeries.
This is a construction which provides a Legendrian embedding of the manifold obtained by surgery on a sphere inside a Legendrian submanifold and, moreover, produces an exact Lagrangian elementary handle-attachment cobordism from the former to the latter Legendrian submanifold. We refer to \cite{LASALCH} for more details.

The isotropic 2-disk $D_2$ with boundary on $\Lambda$  shown in Figure~\ref{nontrcob} is a so-called isotropic surgery disc (since the disc is Lagrangian, there are no choices of framings involved), and it thus determines a Legendrian ambient 1-surgery on $\partial D_ 2 \subset \Lambda$. The resulting Legendrian embedding $\Lambda' \subset \C^2 \times \R$ is diffeomorphic to the 2-sphere obtained by this 1-surgery, and its front-projection is shown in Figure~\ref{nontrcob}.

The isotropic 1-disk $D_1$ with boundary on $\Lambda'$ whose front projection is shown in Figure~\ref{nontrcob}, together with a framing of its symplectic normal bundle, also determines an isotropic surgery 1-disc. The choice of framing is unique if we require that the Legendrian submanifold produced by the corresponding Legendrian ambient surgery has vanishing Maslov class. The resulting Legendrian embedding is the torus $\Lambda''$, which can be seen to be Legendrian isotopic to $\Lambda$. (In this case the Legendrian ambient surgery corresponds to the cusp connected sum as described in \cite{NILSIR}, but performed on a single component).

Joining the exact Lagrangian elementary cobordisms produced by the above two consequtive Legendrian ambient surgeries, together with the exact Lagrangian concordance from $\Lambda''$ to $\Lambda$ induced by the isotopy (see \cite[Theorem 1.1]{LCOLK}), we have produced an exact Lagrangian endocobordism $L$ of $\Lambda$ which can be seen to be diffeomorphic to a solid torus with a neighborhood of a contractible curve removed. In particular, this cobordism is not diffeomorphic to $\R\times T^{2}$. Observe that $L * L$ is diffeomorphic, and even Hamiltonian isotopic, to $L$.

\section{Loose Legendrians}\label{looselegsect}
The theory of loose Legendrian knots in $\C^{n}\times \R$ for $n>1$  and the h-principle for Lagrangian caps with loose Legendrian ends has been recently discovered by Murphy \cite{LLEIHDCM} and then developed by Eliashberg and Murphy \cite{LC}.
\begin{defn}
Let $\Lambda$ be a Legendrian submanifold of $\C^{n}\times \R$. We say that $\Lambda$ is loose if and only if there exists a neighborhood $U$ of $\Lambda$ contactomorphic to a so-called standard loose chart $(R^{n-1}_{abc}, \Lambda_{0})$ with $a<bc$. Here
\begin{align*}
R^{n-1}_{abc} =& \{ (x,y, x_{1},\dots, y_{n-1}, z) : |x|,|y|\leq 1, \|(x_{1},\dots,x_{n-1})\|\leq b,\\ &\|(y_{1},\dots,y_{n-1})\|\leq c, |z|\leq a  )\}\subset \left(\C^{n}\times \R, dz-ydx-\sum_{i}y_{i}dx_{i}\right)
\end{align*}
and $\Lambda_{0}$ is the Legendrian solid cylinder, which is the product of
\begin{align*}
D^{n-1}_{b} = \{ (x_{1},y_{1},\dots,x_{n-1},y_{n-1}) : y_{1} = \dots = y_{n-1} = 0, \|(x_{1},\dots,x_{n-1})\|\leq b \}
\end{align*}
and a Legendrian curve $\lambda_{0}
\subset (\R^{3}, (x,y,z))$ whose front projection is shown in Figure~\ref{looseproj}. The slopes at the self-intersection point of the front are $\pm 1$ and the slope is everywhere in the interval $[-1,1]$, so that the Legendrian arc $\lambda_{0}$ is contained in the box
\begin{align*}
Q_{a}=\{|x|\leq 1, |y|\leq 1, |z|\leq a\}
\end{align*}
with $\partial \lambda_{0}\subset \partial Q_{a}$.
\end{defn}

\begin{figure}[t]
\begin{center}
\labellist
\pinlabel $z$ at 107 120
\pinlabel $a$ at 109 103
\pinlabel $-a$ at 109 8
\pinlabel $x$ at 205 62
\pinlabel $1$ at 187 62
\pinlabel $-1$ at 11 62
\endlabellist
\includegraphics[width=285pt]{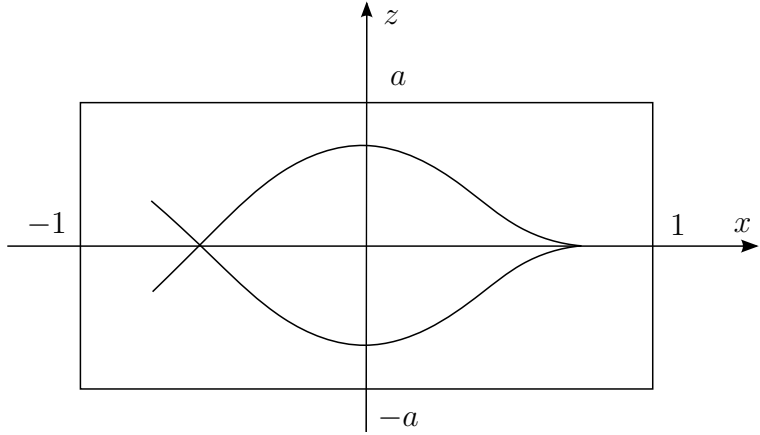}
\caption{The front projection of $\lambda_{0}$.}\label{looseproj}
\end{center}
\end{figure}

We now prove Proposition~\ref{anycoblooselegsphere}.

\subsection{Proof of Proposition~\ref{anycoblooselegsphere}}
Let $\Sigma_g$ be a closed, connected genus $g$
loose Legendrian surface. Consider a trivial exact Lagrangian
cobordism $\R\times \Sigma_g$. Given a closed, orientable, connected
$3$-manifold $M$, from the discussion in \cite{FNDEDDPDDLEP}, see
also \cite{CELIWFDP}, it follows that there exists a self-transverse
exact Lagrangian immersion of $M$ to the symplectization of
$\C^2\times \R$ with $2k-1$ double points for some $k>0$. We take
Legendrian lifts of $M$ and $\R\times \Sigma_g$ in the
contactization of the symplectization $\R\times \C^2\times \R$ and
denote them by $M_{\mathrm{leg}}$ and $(\R\times \Sigma_g)_{\mathrm{leg}}$,
respectively. We now take a cusp connected sum $M_{\mathrm{leg}}\#_{\mathrm{cusp}}
(\R\times \Sigma_g)_{\mathrm{leg}}$ of $M_{\mathrm{leg}}$ and $(\R\times
\Sigma_g)_{\mathrm{leg}}$ as defined in \cite{NILSIR}. Observe that even
though the cusp connected sum construction depends on the cusp edges
of the components and that the notion of Legendrian lift in the
contactization is defined only up to a shift in the contactization
direction, we can assume that the projection of $M_{\mathrm{leg}}\#_{\mathrm{cusp}}
(\R\times \Sigma_g)_{\mathrm{leg}}$ to the symplectization $\R\times
\C^2\times \R$ can be seen as an image of a self-transverse exact
Lagrangian immersion $f_{0}:M\# (\R\times \Sigma_g)\to \R\times
\C^2\times \R$ with $2k$ double points. Finally, we apply Theorem
2.2 from \cite{LC} and see that there exists a compactly supported
isotopy $f_t:M\# (\R\times \Sigma_g)\to \R\times \C^2\times \R$,
$t\in [0,1]$, such that $f_1:M\# (\R\times \Sigma_g)\to \R\times
\C^2\times \R$ defines an embedded exact Lagrangian endocobordism of
$\Sigma_g$ and we denote it by $L:=f_1(M\# (\R\times \Sigma_g))$.

\subsection{Proof of Proposition~\ref{bigspuntriviallooselegspheres}}
Consider the trivial cylindrical cobordism $\R \times \Lambda$ from $\Lambda$ to $\Lambda$, which is of dimension $n+1$. We first perform $N$ number of Legendrian ambient 0-surgeries (i.e. cusp connected sums) on 0-spheres inside $\R \times \Lambda$, producing an exact Lagrangian immersed cobordism diffeomorphic to
\[\R \times \Lambda \# \underbrace{(S^1 \times S^n) \# \hdots \# (S^1 \times S^n}_N)\]
having exactly $N$ transverse double-points. Each such double-point can moreover be seen to have grading $n$ (modulo the Maslov number) and, in the case when $n+1=2\ell$, its Whitney self-intersection index is given by $\sigma=(-1)^{n(n-1)/2+1}$. We then perform $N$ number of Legendrian ambient 1-surgeries on contractible curves in the above cobordism. Each such surgery adds a transverse double-point of grading $n-1$ which, in the case when $n+1=2\ell$, has Whitney self-intersection index $-\sigma$. In order to find the isotropic surgery 2-discs which determine the latter surgery, observe that there exists an isotropic 2-disc in $\C^{n+1}$, $n \ge 2$, whose interior is disjoint from the real-part and whose boundary intersects the real-part orthogonally. This can be seen by first finding an isotropic disc with the correct behavior along the boundary, but whose interior intersects the real-part, and then using a general position argument.

The result in \cite{LC} provides an embedded exact Lagrangian cobordism $L$ from $\Lambda$ to $\Lambda$ which is regular homotopic to the above immersed cobordism. Here we have used the fact that $\Lambda$ is loose, together with the fact that the number of double-points of the immersed cobordism is even and, furthermore, that its total relative self-intersection index is zero in the case $n+1=2\ell$. Finally, observe that $\dim H_{1}(L;\F) \ge N$. This finishes the proof.

\section{Wrapped Floer homology and a version of Seidel's isomorphism}
\label{sec:wrapped}
Here we introduce a version of the wrapped Floer homology complex for an exact Lagrangian filling, defined over an arbitrary ring $R$, that will suit our purposes. We will then prove Seidel's isomorphism (Theorem \ref{thm:isom}) in this setting.

The version of Legendrian contact homology defined in contactization
$(P \times \R,\alpha)$ of an exact symplectic manifold
$(P,d\theta)$, where the differential is defined by a count of
pseudo-holomorphic polygons in $P$, can be defined over $\mathbb{Z}$
by the work in \cite{OILCHAELI} in the case when the Legendrian
submanifold is spin. However, signs are yet to be defined for the
version of Legendrian contact homology whose differential counts
pseudo-holomorphic disks in the symplectization. Consequently, the
wrapped Floer homology as defined in \cite{RSFTLLCHALFC}, together
with the exact sequence in \cite{ANOLCBLSOR}, are also only defined
with coefficients in $\mathbb{Z}_2$ at this point. Instead of
introducing signs into this theory, we will circumvent this issue by
relying on the signs defined in \cite{OILCHAELI}.

Given an $(n+1)$-dimensional exact Lagrangian filling $L_{\Lambda}$ of
$\Lambda$, we consider the following construction. Consider a smooth cut-off function $\rho \co \R
\to \R$ satisfying $\rho'(t) \ge 0$,
$\rho(t)=0$ for $t \le-1$, and $\rho(t)=1$ for $t \ge 0$. Let
$m$ and $M$ denote the minimal and maximal length of the Reeb chords on $\Lambda$, respectively, and let $\Lambda'$ be the Legendrian manifold obtained by translating $\Lambda$ by
$-\epsilon$ in the $z$-coordinate, where $0<\epsilon<m$ is arbitrary. We let $L_{\Lambda'}$ be the exact Lagrangian filling obtained by translating $L_\Lambda$ by $-\epsilon$ in the $z$-coordinate as well, and then perturbing it by a $C^1$-small Hamiltonian
isotopy so that the intersection-points $L_\Lambda \cap
L_{\Lambda'}$ become transverse. For each $N>0$, we define the exact
Lagrangian filling
\[L_\Lambda^{\epsilon,M,N} := \phi^{-M}(L_{\Lambda'}),\]
where $\phi^s \co \R \times (P \times \R) \to \R \times (P \times
\R)$ is the flow of the Hamiltonian vector-field $\rho(t-N)
\partial_z$. Observe that, for generic $L$, $\Lambda$, and $\rho$,
and for $N > 0$ sufficiently large, the intersections $L_\Lambda \cap
L_\Lambda^{\epsilon,M,N}$ outside of the set $\{ t \le N-1\}$ are all transverse
double-points which furthermore are contained in $\{ N-1 \le
t \le N \}$ and correspond bijectively to the Reeb-chords on
$\Lambda$.

Since the exact Lagrangian immersion $L_\Lambda \cup
L_{\Lambda}^{\epsilon,M,N}$ is disconnected, there is a choice involved when constructing a Legendrian
lift to the contactization of the symplectization.
We will choose a lift for which all Reeb chords, which constitute a finite
set, start on the lift of $L_\Lambda$. We will not distinguish
between the union of Lagrangian fillings and of its Legendrian lift,
but for the latter we will implicitly always choose the lift as
above.

In the following we will use $R$ to denote an arbitrary unital ring.
Assuming that $L_\Lambda$ is spin, we fix a spin-structure and a
compatible almost complex structure on the symplectization $\R
\times P \times \R$ which is cylindrical outside of a compact set,
and consider the induced Legendrian contact homology DGA
\[(\mathcal{A} (L_\Lambda \cup L_{\Lambda}^{\epsilon,M,N};R),\partial)\]
with coefficients in $R$ which is induced by the above Legendrian
lift. In the non-spin case, we must take $R=\Z_2$. Recall that the grading of this complex depends on the choice
of a Maslov potential. We refer to \cite{LCHIPR} and \cite{OILCHAELI} for more
details.

The difference between the set-up here and that in \cite{LCHIPR}
is that our Legendrian submanifold is non-compact and that our exact
symplectic manifold has a concave end. For that reason, we need to establish the following compactness result.
\begin{lem}
\label{lem:compactness}
Let $(\overline{P},d\theta)$ be a compact exact symplectic manifold with boundary and let
$J_s$, $s \in [0,1]$, be a smooth family of compatible almost complex
structures on $(\R \times \overline{P} \times \R,d(e^t\alpha))$ coinciding with a
fixed cylindrical almost complex structure outside of a compact set.
We moreover require that $\R \times \partial\overline{P} \times \R$ is
$J_s$-convex for all $s \in [0,1]$.

Suppose that
\[L_{\Lambda}, L_{\Lambda'}^s \subset \R \times \mathrm{int}\overline{P} \times \R, \:\: s\in[0,1],\]
are exact Lagrangian fillings, where the family $L_{\Lambda'}^s$ depends smoothly on $s$ and is fixed outside of a
compact set, and where $\Lambda \cup \Lambda'$ is embedded and front generic. There is a compact subset of $\R \times
\mathrm{int}\overline{P} \times \R$ that contains all
$J_s$-holomorphic disks in $\R \times \mathrm{int}\overline{P}
\times \R$, $s \in [0,1]$, having
\begin{itemize}
\item boundary on $L_\Lambda \cup L_{\Lambda'}^s$,
\item bounded image,
\item exactly one positive puncture, and
\item finite $d(e^t\alpha)$-energy.
\end{itemize}
\end{lem}
\begin{proof}
First, observe that no $J_s$-holomorphic curve can be tangent to $\R
\times \partial \overline{P} \times \R$ from the inside since the latter
hypersurface is $J_s$-convex. It thus suffices to show that a
pseudo-holomorphic curve as in the assumption has $t$-coordinate
satisfying the bound $|t|<C$ for some constant $C$ independent of $s \in
[0,1]$.

By contradiction, we assume that there is a sequence $u_n$ of
$J_s$-holomorphic disks satisfying the properties of the assumption
for which $(t_n,p_n) \in \mathrm{im}(u_n)$, where $t_n \to
\pm\infty$.

We start by establishing a bound on the Hofer energy of the $J_s$-holomorphic disks $u$ as in
the assumption. Take a real number $A$ such that $J_s$ all are cylindrical in $\{ t
\le A\}$ and $L_\Lambda \cup L_{\Lambda'}^s \subset \{ t \ge A \}$.
Consider a smooth function $\varphi(t) \ge 0$ satisfying
$\varphi'(t) \ge 0$, and $\varphi(t)=e^t$ for $t \ge A$. The fact
that $e^t\alpha$ is exact when restricted to the fillings implies
there is a constant $D$ independent of $s \in [0,1]$ for which the inequality
\[ \int_u d(\varphi(t)\alpha) \le D\]
is satisfied for any $J_s$-holomorphic $u$ as in the assumption.

The above energy bound implies that the compactness theorem for pseudo-holomorphic curves
in symplectic field theory \cite{CRISFT} can be applied. This
rules out the possibility that $t_n \to -\infty$, since such a
sequence of $u_n$ would have a sub-sequence converging to a
pseudo-holomorphic building consisting of at least two levels,
contradicting the fact that there are no periodic Reeb orbits on
$(\overline{P} \times \R,dz+\theta)$.

To rule out the case $t_n \to +\infty$ we argue as follows. There is
a bound on $d(e^t\alpha)$-energy independent of $s\in[0,1]$ for any
$J_s$-holomorphic disk $u$ as in the assumption. The monotonicity
property for the $d(e^t\alpha)$-area of pseudo-holomorphic disks
with boundary \cite[Propositions 4.3.1 and 4.7.2]{SPOHCIACM})
implies that the $d(e^t\alpha)$-area of $u_n$ converges to $+\infty$,
which leads to a contradiction.
\end{proof}

Recall that, by assumption, there is a compact domain $\overline{P}\subset P$ with $J_P$-convex boundary, for some compatible almost complex structure $J_P$ on $(P,d\theta)$, such that $L_\Lambda \cup  L_{\Lambda}^{\epsilon,M,N} \subset \R \times \mathrm{int}\overline{P} \times \R$. Assume that the compatible almost complex structure $J$ on $\R \times \overline{P} \times \R$ satisfies the assumptions of Lemma \ref{lem:compactness}, as well as the integrability condition specified in \cite{LCHIPR} in some neighbourhood of the double points $L_\Lambda \cap  L_{\Lambda}^{\epsilon,M,N}$. We will call such a choice of almost complex structure \emph{admissible}. For example, one can use a perturbation of a cylindrical almost complex structure $J$ on $\R \times P \times \R$ for which the canonical projection $\R \times P \times \R \to P$ is $(J,J_P)$-holomorphic (see e.g.~\cite{LPPTTSOPTRAA}).

For an admissible almost complex structure, the consequences of Lemma \ref{lem:compactness} implies that \cite[Lemma 2.5]{LCHIPR} and \cite[Proposition 2.6]{LCHIPR} can be applied in this non-compact setting as well. In particular, we have $\partial^2=0$, and the homotopy-type
of the DGA is invariant under compactly supported Legendrian
isotopies as well as compactly supported deformations of the
admissible almost complex structure.

We will however not consider the full DGA, but only its ``linear
part''. Since each component of $L_\Lambda \cup L_{\Lambda}^{\epsilon,M,N}$
is embedded, the above differential $\partial$ respects the
word-length filtration of the DGA and thus descends to a
``linearized differential'' on the $R$-module
\[(CF_i(L_\Lambda,L_\Lambda^{\epsilon,M,N};R),\partial)\]
freely generated by the double points $L_\Lambda \cap L_\Lambda^{\epsilon,M,N}$ (recall that there is a bijective correspondence between these double points and the Reeb chords on the Legendrian lift). We will consider the induced co-complex
\[(CF^i(L_\Lambda,L_{\Lambda}^{\epsilon,M,N};R),d).\]
Given double points $a,b$, the coefficient in front of $b$ in the expression
$d(a)$ is thus defined to be the signed count of rigid $J$-holomorphic
strips
\begin{gather*}
u \co \R \times [0,1] \to \R \times P \times \R,\\
du +Jdu\circ i=0,
\end{gather*}
having bounded image and satisfying
\[\begin{cases}u(s,0) \in L_\Lambda, \:\: u(s,1) \in L_\Lambda^{\epsilon,M,N},\\
\lim_{s \to +\infty} u(s,t)=b, \:\: \lim_{s \to -\infty} u(s,t)=a,\\
\int_u d(e^t\alpha)<\infty.
\end{cases}
\]
We refer to \cite[Formula (2.2)]{LCHIPR} for more details.

\begin{rem}
For coefficients $R=\Z_2$, this complex coincides with the wrapped
Floer homology complex in \cite{RSFTLLCHALFC} and
\cite{LPPTTSOPTRAA}.
\end{rem}

We have the following invariance result.
\begin{prop}[\cite{LCHIPR}]
\label{prop:inv} $(CF^i(L_\Lambda,L_{\Lambda}^{\epsilon,M,N};R),d)$ defines
a complex whose homotopy-type is invariant under compactly supported
deformations of an admissible almost complex structure as well as of
compactly supported Hamiltonian isotopies. In partcular, it is
null-homotopic for any choice of spin structure and compatible
almost complex structure as above.
\end{prop}
\begin{proof}
Using Lemma \ref{lem:compactness} it follows that \cite[Lemma 2.5, Proposition 2.6]{LCHIPR} can be applied. This shows that
$d^2=0$ holds together with the above invariance property.

To show that the complex is null-homotopic, we argue as follows.
There is a Hamiltonian vector-field
\[X:=(\rho(t-N)-1)\partial_z\]
with $\rho$ as above, whose time-$T$ flow has the property that
\[L_\Lambda \cap  \phi^T_X(L_\Lambda^{\epsilon,M,N}) = \emptyset,\]
for $T>0$ sufficiently large.
The invariance result now implies the statement, since this isotopy
of $L_\Lambda^{\epsilon,M,N}$ may be induced by a compactly supported
Hamiltonian isotopy, as follows by a standard argument.
\end{proof}

Let $CF^i_\infty(L_\Lambda,L_{\Lambda}^{\epsilon,M,N}) \subset
CF^i(L_\Lambda,L_{\Lambda}^{\epsilon,M,N})$ be the sub-module generated by
the double-points corresponding to the Reeb chords on $\Lambda \subset P \times \R$. Recall
that these all are contained in the set $\{ N-1 \le t \le N \}$.
\begin{lem}
For $N>0$ large enough,
\[CF^i_\infty (L_\Lambda,L_\Lambda^{\epsilon,M,N};R) \subset (CF^i(L_\Lambda,L_{\Lambda}^{\epsilon,M,N};R),d)\]
is a sub-complex.
\end{lem}
\begin{proof}
Let $a,b \in L_\Lambda \cap L_{\Lambda}^{\epsilon,M,N}$
be double points, where $b$ is contained in the complement of $\{ t \ge N-1 \}$
and $a$ corresponds to a Reeb chord on $\Lambda$. We will show that
the coefficient of $b$ vanishes in the expression $d(a)$.

Recall that this coefficient is given by a count of $J$-holomorphic strips $D\subset \R \times P \times \R$ with boundary on $L_\Lambda \cup L_{\Lambda}^{\epsilon,M,N}$ whose $d(e^t\alpha)$-area can be computed by
\[0<\int_D d(e^t\alpha) =\int_{\partial{D}} e^t\alpha=(f(b)-g_N(b))-(f(a)-g_N(a)),\]
where $f$ and $g_N$ are the primitives of $e^t\alpha$ pulled back to $L_\Lambda$ and $L_{\Lambda}^{\epsilon,M,N}$, respectively. The claim follows since, after increasing $N$, we may assume that $g_N(b)=0$ while $g_N(a)<0$ is arbitrarily small.
\end{proof}

We shall write
\[ CF^i_\infty (L_\Lambda;R) := CF^i_\infty (L_\Lambda,L_\Lambda^{\epsilon,M,N};R),\]
given that $N>0$ is sufficiently big, so that the above lemma can be applied. Note that this complex is generated by the Reeb chords on $\Lambda$, but that the differential depends on the choice of almost complex structure, the numbers $\epsilon,M,N>0$, as well as the Hamiltonian perturbation used in the construction of $L_\Lambda^{\epsilon,M,N}$.

\begin{proof}[Proof of Theorem \ref{thm:isom}]
After an appropriate compactly supported Hamiltonian isotopy, we may
suppose that there is a Weinstein neighborhood of $L_\Lambda \cap \{
t \le N -1 \}$ isomorphic to a neighborhood of the zero-section of
$(T^*(L_\Lambda \cap \{ t \le N-1 \}),d\theta_{L_\Lambda})$ (where
$\theta_{L_\Lambda}$ denotes the Liouville form), where $L_\Lambda$
is identified with the zero-section and $L_\Lambda^{\epsilon,M,N}
\cap \{ t \le N -1 \}$ is given as a graph $-df$. To that end,
observe that $L_\Lambda^{\epsilon,M,N}$ can be assumed to be
arbitrarily $C^1$-close to $L_\Lambda$ on this set by construction,
given that we choose $\epsilon>0$ to be sufficiently small. We may
furthermore suppose that $f\co L_\Lambda \cap \{ t \le N -1 \} \to
\R$ is a Morse-function and that $df$ evaluates negatively on any
outward-pointing normal to the boundary $L_\Lambda \cap \{ t=N-1\}$.

Let $g$ be a Riemannian metric on $L_\Lambda \cap \{ t \le N-1\}$ for which $(f,g)$ is Morse-Smale. Assume that $f$ is sufficiently small, and that the choice of an admissible almost complex structure is induced by the metric $g$ in some neighborhood of $L_\Lambda \cap \{ t \le N-1\}$ as in \cite[Remark 6.1]{ADESFLCH}. The analytical result in \cite[Lemma 6.11]{ADESFLCH} can be applied to show that
\[(CF^i(L_\Lambda,L_\Lambda^{\epsilon,M,N};R),d) / CF^i_\infty (L_\Lambda,L_\Lambda^{\epsilon,M,N};R) = (C^{\mathrm{Morse}}_{n-i} (f;R),d_f),\]
for generic such choices together with an appropriate choice of Maslov potential. Here, the latter complex is the Morse complex associated to the Morse-Smale pair $(f,g)$.

We refer to \cite[Proposition 3.7(2)]{ADESFLCH} for such a result in the closed case. In the current non-compact setting one must use the fact that, after choosing an appropriate function $f$ as above, we may assume that all relevant pseudo-holomorphic strips are contained inside some arbitrarily small neighborhood of $L_\Lambda \cap \{ t \le N -1\}$. To that end, we might have to make the critical values of $f$ arbitrarily small, while keeping the function fixed in some neighborhood of the boundary $L_\Lambda \cap \{ t = N-1\}$. The sought behavior then readily follows from the monotonicity property for the symplectic area of the relevant pseudo-holomorphic strips \cite{SPOHCIACM}. We also refer to the proof of \cite[Theorem 6.2]{LPPTTSOPTRAA}, where a similar argument is given in detail.

The action-filtration implies that we can write
\begin{gather*} (CF(L_\Lambda, L_\Lambda^{\epsilon,M,N};R),d) = \mathrm{Cone}(\Phi), \\
\Phi \co C^{\mathrm{Morse}}_{n-i} (f;R) \to CF^{i+1}_\infty (L_\Lambda,
L_\Lambda^{\epsilon,M,N};R),
\end{gather*}
as a mapping-cone, where $\Phi$ is a chain-map of degree zero
induced by the differential $d$.

Proposition \ref{prop:inv} shows that the above mapping cone is
acyclic. The map $\Phi$ thus produces the required isomorphism.
\end{proof}

\appendix

\section{Spherical spinning preserves looseness}\label{fappthfrspconstrprlnss}

Here we prove that the front $S^{m}$-spinning construction preserves
looseness. This is natural since, by construction, the $S^m$-spun of
$\Lambda \subset \C^n \times \R$ has a neighbourhood contactomorphic
to $\C^n \times T^*S^m \times \R$, where $\Sigma_{S^m} \Lambda$ is
identified with the product $\Lambda \times 0_{S^m}$. If there is a
neighborhood in $\C^n \times \R$ which intersects $\Lambda$ in a
so-called loose chart, it is then readily seen that a loose chart for $\Lambda \times 0_{S^m}\subset \C^n \times T^*S^m \times \R$ can be produced by taking a suitable
product of neighborhoods. To that end, it is important to notice that $T^*S^m$ has
infinite symplectic area.

\subsection*{Proof of Proposition~\ref{spunlooselegisloose}}

Let $\Lambda\subset \C^n\times \R$ be parametrized by
$f_{\Lambda}: \Lambda \to \C^{n}\times\R$ with
\begin{align*}
f_{\Lambda}(p) = (x(p),y(p),x_{1}(p),y_{1}(p),\dots,x_{n-1}(p),y_{n-1}(p),z(p)).
\end{align*}
Without loss of generality assume that $x_{n-1}(p)>0$ for all $p$. Hence, $\Lambda\subset J^{1}(\R^{n-1}\times\R_{+})\subset J^1(\R^n)$.
Apply the front $S^m$-spinning construction to $\Lambda$, producing $\Sigma_{S^m}\Lambda$ whose front-projection $\Pi_F(\Sigma_{S^m}\Lambda)$ is parametrized by
\begin{align*}
\Pi_{F}\circ f_{\Sigma_{S^m}\Lambda}(p,\theta,\bar{\phi}) = (x(p), x_{1}(p), \dots, x_{n-1}(p)\cos \phi_{m-1}, \dots, x_{n-1}(p)\sin \theta\dots \sin \phi_{m-1},z(p)),
\end{align*}
where $\theta\in [0,2\pi)$ and $\bar{\phi} = (\phi_{1},\dots,\phi_{m-1})\in [0,\pi]^{m-1}$.

Consider  $i: S^m \times \R^{n-1}\times \R_{+} \to \R^{n+m}$ given by
\begin{align*}
i(\theta, \bar{\phi}, x, x_{1}, \dots,x_{n-2},x_{n-1}) = (x, x_{1}, \dots, x_{n-2}, x_{n-1}\cos \phi_{m-1}, \dots, x_{n-1}\sin \theta\dots \sin \phi_{m-1}),
\end{align*}
where $\theta\in [0,2\pi)$ and $\bar{\phi} = (\phi_{1},\dots,\phi_{m-1})\in [0,\pi]^{m-1}$. Note that $i$ is a diffeomorphism onto its image.
This map can be naturally extended to the bundle isomorphism
\begin{align*}
i_{\mathrm{sym}}:=(i^{-1})^{\ast}: T^{\ast} (S^m \times \R^{n-1}\times
\R_{+})\to T^{\ast} (i(S^m\times \R^{n-1}\times \R_{+}))\subset
T^*(\R^{n+m})
\end{align*}
such that $i_{\mathrm{sym}}^{\ast}(\theta_{\R^{n+m}}|_{i(S^m \times \R^{n-1}\times \R_{+})}) = \theta_{S^m}+\theta_{\R^{n-1}\times \R_{+}}$, where $\theta_{\R^{n+m}}$ is the tautological 1-form on $T^{\ast} (\R^{n+m})$, $\theta_{S^m}$ is the tautological 1-form on $T^{\ast} (S^m)$ and $\theta_{\R^{n-1}\times \R_{+}}$ is the tautological 1-form on $T^{\ast}(\R^{n-1}\times \R_{+})$. Then we extend $i_{\mathrm{sym}}$ to the map
\begin{align*}
i_{\mathrm{cont}}: (J^{1}(S^m \times \R^{n-1}\times \R_{+}), dz +\theta_{S^m}+\theta_{\R^{n-1}\times \R_{+}}) \to (J^{1}(\R^{n+m}), dz + \theta_{\R^{n+m}})
\end{align*}
in the natural way, i.e., $i_{\mathrm{cont}}$ maps $z$ to $z$. Observe that $i_{\mathrm{cont}}$ is a contactomorphism onto its image.

From the parametrization of $\Pi_{F}(\Sigma_{S^m}\Lambda)$ it follows that $\Sigma_{S^m}\Lambda$ is a Legendrian submanifold of $i_{\mathrm{cont}}(J^{1}(S^m \times \R^{n-1}\times \R_{+}))$.
We define $\widetilde{\Lambda}$ to be the Legendrian submanifold of $J^{1}(S^m \times \R^{n-1}\times \R_{+})$ parametrized by
$i_{\mathrm{cont}}^{-1}\circ f_{\Sigma_{S^m}\Lambda}(p,\theta,\bar{\phi})$ with $p\in \Lambda$, $\theta\in [0,2\pi)$ and $\bar{\phi} = (\phi_{1},\dots,\phi_{m-1})\in [0,\pi]^{m-1}$. Observe that $\widetilde{\Lambda}$  is a product of the $0$-section $0_{S^m}$ of $T^{\ast} S^m$ and $\Lambda\subset J^1(\R^{n-1}\times\R_{+})$
(here we decompose $J^{1}(S^m \times \R^{n-1}\times \R_{+})$ as $T^{\ast}S^{m}\times J^{1}(\R^{n-1}\times\R_{+})$).

We now show that $\Sigma_{S^{m}}\Lambda$ is loose. Since $\Lambda$ is loose, there exists a neighborhood $U$ of $\Lambda$ in $J^{1}(\R^{n-1}\times\R_{+})$ which is contactomorphic to a standard loose chart $(R^{n-1}_{abc}, \Lambda_{0})$, where
$a<bc$. Without loss of generality assume that $b=1/2$ (we can make this assumption since $(R^{n-1}_{abc}, \Lambda_{0})$ can always be rescaled). Consider $\widetilde{V}:=T^{\ast} S^m\times U$. From the definition of $\widetilde{\Lambda}$ it follows that  $\widetilde{V}\cap \widetilde{\Lambda} = 0_{S^m}\times (U\cap \Lambda)$.
Let $\phi$ denote the contactomorphism
\begin{align*}
(U, U\cap \Lambda, dz + \theta_{\R^{n-1}\times \R_{+}})\to \left(R^{n-1}_{a\frac{1}{2}c}, \Lambda_{0}, dz-ydx-\sum_{i}y_{i}dx_{i}\right)
\end{align*}
and assume that
$\phi^{\ast}(dz-ydx-\sum_{i}y_{i}dx_{i}) = f(s)(dz + \theta_{\R^{n-1}\times \R_{+}})$, where $f\in \C^{\infty}(U,\R)$ with $f(s)\neq 0$ for all $s\in U$. Consider $(\widetilde{V}, \widetilde{V}\cap \widetilde{\Lambda}) = (T^{\ast} S^m\times U, 0_{S^m}\times (U\cap \Lambda))$ equipped with the contact form $dz + \theta_{\R^{n-1}\times \R_{+}} + \theta_{S^m}$, where $\theta_{S^m}=- \sum_{j} p_{j}dq_{j}$. In addition, consider $T^{\ast}S^{m}\times R^{n-1}_{a\frac{1}{2}c}$ with the contact form $dz  - ydx-\sum_{i}y_{i}dx_{i} - \sum_{j} p_{j}dq_{j}$.
We now define the map
$\widetilde{\phi}:\widetilde{V}\to T^{\ast}S^{m}\times R^{n-1}_{a\frac{1}{2}c}$
by setting $\widetilde{\phi}(\bar{q},\bar{p},s)=(\bar{q}, f(s)\bar{p},\phi(s))$, where $s\in U$ and $(\bar{q},\bar{p})\in T^{\ast}S^m$. Since $\phi$ is a diffeomorphism and $f(s)\neq 0$ for all $s\in U$, $\widetilde{\phi}$ is also a diffeomorphism. In addition,
\begin{align*}
\widetilde{\phi}^{\ast}\left(dz - ydx-\sum_{i}y_{i}dx_{i} - \sum_{j} p_{j}dq_{j}\right) = f(\bar{q},\bar{p},s)\left(dz + \theta_{\R^{n-1}\times \R_{+}} - \sum_{j} p_{j}dq_{j}\right)
\end{align*}
for $f(\bar{q},\bar{p},s):=f(s)$, where $s\in U$, $(\bar{q},\bar{p})\in T^{\ast}S^m$, and hence $\widetilde{\phi}$ is a contactomorphism. Finally, observe that $\widetilde{\phi}(\widetilde{V}\cap \widetilde{\Lambda})=\widetilde{\phi}(0_{S^m}\times (U\cap \Lambda))=0_{S^m}\times \Lambda_{0}$.
In conclusion, we get that
$(\widetilde{V}, \widetilde{V}\cap \widetilde{\Lambda}, dz + \theta_{\R^{n-1}\times \R_{+}} - \sum_{j} p_{j}dq_{j})$ is contactomorphic to
\[\left(T^{\ast}S^{m}\times R^{n-1}_{a\frac{1}{2}c}, 0_{S^m}\times \Lambda_{0},dz-ydx-\sum_{i}y_{i}dx_{i}-\sum_{j}p_{j}dq_{j}\right).\]
From the definition of $R^{n-1}_{a\frac{1}{2}c}$ it follows that we can write it as $B^1_1\times B^1_1 \times B^1_a\times B^{n-1}_{\frac{1}{2}}\times B^{n-1}_c$, where $B^{k}_{d}$ is a closed $k$-ball of radius $d$ and all the balls in the decomposition are centered at $0$.
Observe that there exists
\[B^{m}_{1/2}\times B^{m}_c\subset \left(T^{\ast}D^{m}, d\left(-\sum_{k}\widetilde{y}_{k}d\widetilde{x}_{k}\right)\right)\hookrightarrow \left(T^{\ast}S^m, d\left(- \sum_{j} p_{j}dq_{j}\right)\right),\]
 where $D^{m}$ is a closure of one of the hemispheres of $S^{m}$ and $B^{m}_c$ is centered at $0$. Hence, there is some $V$ such that $(V\cap (T^{\ast}S^{m}\times R^{n-1}_{a\frac{1}{2}c}), V\cap (0_{S^m}\times \Lambda_{0}))$ can be written as
\begin{align*}
(B^{m}_{\frac{1}{2}}\times B^{m}_c\times B^1_1\times B^1_1 \times B^1_a\times B^{n-1}_{\frac{1}{2}}\times B^{n-1}_c, B^{m}_{\frac{1}{2}}\times \{0\}\times \Lambda_{0}),
\end{align*}
where all the balls, possibly except of $B^{m}_{\frac{1}{2}}$, are centered at $0$. Thus, there is a neighborhood $\widetilde{U}$ such that $(\widetilde{U}\cap J^{1}(S^m \times \R^{n-1}\times \R_{+}), \widetilde{U}\cap \widetilde{\Lambda})$ is contactomorphic to
$(R^{n+m-1}_{a\frac{1}{2}c}, \lambda_{0}\times D^{n+m-1}_{\frac{1}{2}})$. Hence, there is a neighborhood of $\Sigma_{S^m}\Lambda$ contactomorphic to a standard loose chart.
In conclusion, $\Sigma_{S^m}\Lambda$ is a loose Legendrian submanifold of $\C^{n+m}\times \R$.

\begin{remark}
Observe that the proof of Proposition~\ref{spunlooselegisloose} works not only for $\Sigma_{S^m} \Lambda$ but also for $\Lambda'\subset J^1(\R^{n+m})$ such that there exists a contactomorphism onto its image
\begin{align*}
(T^{\ast}K\times J^1(\R^{n}), dz + \theta_{\R^n} + \theta_{K})\hookrightarrow (J^1(\R^{n+m}), dz + \theta_{\R^{n+m}})
\end{align*}
 which maps $0_{K}\times \Lambda$ to $\Lambda'$. Here $K$ is a closed $m$-dimensional manifold, $0_{K}$ is the $0$-section of $T^{\ast}K$, $\theta_{\R^{n}}$ is the tautological 1-form on $T^{\ast} \R^n$, $\theta_{\R^{n+m}}$ is the tautological 1-form on $T^{\ast}\R^{n+m}$ and $\theta_{K}$ is the tautological 1-form on $T^{\ast} K$.
\end{remark}

\section{Forgetfulness of the front $S^{m}$-spinning construction}\label{forgetfrontsmconstrex}

In this section we provide four types of examples of Legendrian
submanifolds $\Lambda_{1},\Lambda_{2}\subset \C^n\times \R$ which
are not Legendrian isotopic, but which become Legendrian isotopic
after we apply the front $S^m$-spinning construction. Observe that
two Legendrian embeddings in $\C^n \times \R$ are always smoothly
isotopic whenever $n \ge 2$, by a classical result of Haefliger,
since the codimension is sufficiently high in these cases.

\subsection*{Example A}
Let $\Lambda_1$ and $\Lambda_2$ be two Legendrian knots with the
same rotation numbers and Thurston-Bennequin numbers. Assume that
$\Lambda'_1$ is not Legendrian isotopic to $\Lambda'_2$, where
$\Lambda'_i$ is a positive (or negative) stabilization of
$\Lambda_i$, which means that it has a local situation described in
Figure~\ref{looseproj}, $i=1,2$. From the proof of
Proposition~\ref{spunlooselegisloose} it follows that $\Sigma_{S^m}
\Lambda'_i$ is loose for $i=1,2$. If $m$ is even, using the fact
that $\Sigma_{S^m} \Lambda'_1$ and $\Sigma_{S^m} \Lambda'_2$ are two
loose embeddings of $S^1\times S^m$ with the same classical
invariants, Theorem A.4 in \cite{LLEIHDCM} together with the
h-principle for loose Legendrian knots \cite{LLEIHDCM} imply that
$\Sigma_{S^m} \Lambda'_1$ is Legendrian isotopic to $\Sigma_{S^m}
\Lambda'_2$. Observe that it is easy to construct examples of
Legendrian knots with the same classical invariants, that are in
different smooth isotopy classes and whose stabilizations are not Legendrian isotopic (not even smoothly isotopic).
However, a result of Etnyre and
Honda \cite{OCSALK} provides examples of $\Lambda_1$ and $\Lambda_2$ that have the same underlying
smooth knot and the same classical invariants, but for which $\Lambda'_1$ and $\Lambda'_2$ are not Legendrian isotopic. For
these examples \cite[Theorem A.4]{LLEIHDCM} gives that $\Lambda'_1$
is formally isotopic to $\Lambda'_2$ (this is simply the statement
that two such knots are isotopic using only Reidemeister moves of
type II and type III). In this case, since $\Sigma_{S^m} \Lambda'_1$
thus is formally isotopic to $\Sigma_{S^m} \Lambda'_2$, it now
follows by \cite{LLEIHDCM} that $\Sigma_{S^m} \Lambda'_1$ is
Legendrian isotopic to $\Sigma_{S^m} \Lambda'_2$ for all $m$.

\subsection*{Example B}
Let $\Lambda_1$ and $\Lambda_2$ be two Legendrian
knots with the same rotation numbers but different
Thurston-Bennequin numbers. As in Example A, we apply positive (or
negative) stabilization to both $\Lambda_1$ and $\Lambda_2$ to
produce a local situation described in Figure~\ref{looseproj}. We will use
$\Lambda'_i$ to denote the stabilization of $\Lambda_i$, $i=1,2$. Observe
that $r(\Lambda'_i)=r(\Lambda_i)+1$ (or
$r(\Lambda'_i)=r(\Lambda_i)-1$), $\mathrm{tb}(\Lambda'_i)=\mathrm{tb}(\Lambda_i)-1$ for $i=1,2$. Then
we apply the front $S^1$-spinning to $\Lambda'_i$, $i=1,2$. From the
properties of the front $S^1$-spinning construction it follows that
$\Sigma_{S^1} \Lambda'_1$, $\Sigma_{S^1}
\Lambda'_2$ are two embeddings of $S^1\times S^1$, $r(\Sigma_{S^1} \Lambda'_1) = r(\Sigma_{S^1}
\Lambda'_2)$ and, using the result from~\cite{TCOSMODGOET},
$\mathrm{tb}(\Sigma_{S^1}\Lambda'_i)=0$ for $i=1,2$. Hence, from Theorem A.4
and the  h-principle for loose Legendrian knots  from
\cite{LLEIHDCM} it follows that $\Sigma_{S^1} \Lambda'_1$ is
Legendrian isotopic to $\Sigma_{S^1} \Lambda'_2$.

\subsection*{Example C}
Let $\Lambda_1$ and $\Lambda_2$ be two Legendrian
knots with the same rotation numbers but different
Thurston-Bennequin numbers. We construct $\Lambda'_1$ and
$\Lambda'_2$ the same way as in Examples A and B. Again, note that
$r(\Lambda'_i)=r(\Lambda_i)+1$ (or $r(\Lambda'_i)=r(\Lambda_i)-1$),
$\mathrm{tb}(\Lambda'_i)=\mathrm{tb}(\Lambda_i)-1$ for $i=1,2$. Then we apply the front
$S^{2k+1}$-spinning to $\Lambda'_i$, where $i=1,2$ and $k\in \mathbb
N$. From the properties of the front $S^{2k+1}$-spinning
construction it follows that $\Sigma_{S^{2k+1}} \Lambda'_1$ and $\Sigma_{S^{2k+1}} \Lambda'_2$ are two embeddings of $S^1\times S^{2k+1}$,
$r(\Sigma_{S^{2k+1}} \Lambda'_1) = r(\Sigma_{S^{2k+1}} \Lambda'_2)$
and, using the result from \cite{TCOSMODGOET},
\begin{align*}
\mathrm{tb}(\Sigma_{S^{2k+1}}\Lambda'_i)=(-1)^{k+2}\frac{1}{2}\chi(S^1\times
S^{2k+1}) = 0,
\end{align*}
where $i=1,2$. So, we see that $\Sigma_{S^{2k+1}}\Lambda'_1$ and
$\Sigma_{S^{2k+1}}\Lambda'_2$
are two embeddings of $S^1\times S^{2k+1}$ and have the same classical invariants. Finally,  we apply the front
$S^{2l+1}$-spinning to $\Sigma_{S^{2k+1}}\Lambda'_i$ for $i=1,2$ and
$l\in \mathbb N$. From the properties of the front
$S^{2l+1}$-spinning construction it follows that
$\Sigma_{S^{2l+1}}\Sigma_{S^{2k+1}}\Lambda'_1$,
$\Sigma_{S^{2l+1}}\Sigma_{S^{2k+1}}\Lambda'_2$ are two embeddings of $S^1\times S^{2k+1}\times S^{2l+1}$ and they have the
same classical invariants. Hence, we use Theorem A.4 and the
h-principle for loose Legendrian knots from \cite{LLEIHDCM} and get
that $\Sigma_{S^{2l+1}}\Sigma_{S^{2k+1}}\Lambda'_1$ is Legendrian
isotopic to $\Sigma_{S^{2l+1}}\Sigma_{S^{2k+1}}\Lambda'_2$ for
$k,l\in \mathbb N$.

\subsection*{Example D}
This example is an extension of  Example C to high dimensions.
Let $\Lambda_1$ and $\Lambda_2$ be two loose Legendrian
submanifolds of $\C^{2m+1}\times \R$ which are embeddings of an $(2m+1)$-dimensional manifold $\Lambda$ and such that they have the same rotation classes but different
Thurston-Bennequin numbers. We apply the front
$S^{2k+1}$-spinning to $\Lambda_i$, where $i=1,2$ and $k\in \mathbb
N$. From the properties of the front $S^{2k+1}$-spinning
construction it follows that $\Sigma_{S^{2k+1}} \Lambda_1$ and $\Sigma_{S^{2k+1}} \Lambda_2$ are two embeddings of $\Lambda\times S^{2k+1}$,
$r(\Sigma_{S^{2k+1}} \Lambda_1) = r(\Sigma_{S^{2k+1}} \Lambda_2)$
and, using the same argument as in Example C,
$\mathrm{tb}(\Sigma_{S^{2k+1}}\Lambda_i)=0$ for $i=1,2$. So, we see that $\Sigma_{S^{2k+1}}\Lambda_1$ and
$\Sigma_{S^{2k+1}}\Lambda_2$ are embeddings of $\Lambda\times S^{2k+1}$ and have the same classical invariants.
Then
we apply the front
$S^{2l+1}$-spinning to $\Sigma_{S^{2k+1}}\Lambda_i$ for $i=1,2$ and
$l\in \mathbb N$. From the properties of the front
$S^{2l+1}$-spinning construction it follows that
$\Sigma_{S^{2l+1}}\Sigma_{S^{2k+1}}\Lambda_1$ and
$\Sigma_{S^{2l+1}}\Sigma_{S^{2k+1}}\Lambda_2$ are embeddings of $\Lambda\times S^{2k+1}\times S^{2l+1}$ and they have the
same classical invariants. Hence, we use Theorem A.4 and the
h-principle for loose Legendrian knots from \cite{LLEIHDCM} and get
that $\Sigma_{S^{2l+1}}\Sigma_{S^{2k+1}}\Lambda_1$ is Legendrian
isotopic to $\Sigma_{S^{2l+1}}\Sigma_{S^{2k+1}}\Lambda_2$ for
$k,l\in \mathbb N$.

\section*{Acknowledgements}
The authors are deeply grateful to Fr\'{e}d\'{e}ric Bourgeois,
Baptiste Chantraine, Yakov Eliashberg, Paolo Ghiggini, Vera V\'{e}rtesi and Hiro Lee
Tanaka for helpful conversations and interest in their work. In
addition, the authors would like to thank Paolo Ghiggini for
pointing out the way to simplify the proof of
Lemma~\ref{thomisomorphismmaininequality}. Finally, the authors are grateful to the
referee of an earlier version of this paper for many valuable comments and suggestions.

\end{document}